\documentclass[12pt]{amsart}
\usepackage{cases}
\usepackage{mathrsfs}
\usepackage{color}
\usepackage{latexsym,amssymb}
\usepackage{a4wide}
\usepackage{amsthm}
\usepackage{amsmath}
\usepackage{graphicx}
\input xy
  \xyoption{all}
\usepackage{verbatim}
\usepackage[lite,abbrev,msc-links]{amsrefs}

\numberwithin{equation}{section}
\begin{document}
\theoremstyle{plain}
\newtheorem{thm}{Theorem}[section]
\newtheorem{lem}[thm]{Lemma}
\newtheorem{cor}[thm]{Corollary}
\newtheorem{cor*}[thm]{Corollary*}
\newtheorem{prop}[thm]{Proposition}
\newtheorem{prop*}[thm]{Proposition*}
\newtheorem{conj}[thm]{Conjecture}
\theoremstyle{definition}
\newtheorem{construction}{Construction}
\newtheorem{notations}[thm]{Notations}
\newtheorem{question}[thm]{Question}
\newtheorem{prob}[thm]{Problem}
\newtheorem{rmk}[thm]{Remark}
\newtheorem{remarks}[thm]{Remarks}
\newtheorem{defn}[thm]{Definition}
\newtheorem{claim}[thm]{Claim}
\newtheorem{assumption}[thm]{Assumption}
\newtheorem{assumptions}[thm]{Assumptions}
\newtheorem{properties}[thm]{Properties}
\newtheorem{exmp}[thm]{Example}
\newtheorem{comments}[thm]{Comments}
\newtheorem{blank}[thm]{}
\newtheorem{observation}[thm]{Observation}
\newtheorem{defn-thm}[thm]{Definition-Theorem}
\newtheorem*{Setting}{Setting}

\newcommand{\sA}{\mathscr{A}}
\newcommand{\sB}{\mathscr{B}}
\newcommand{\sC}{\mathscr{C}}
\newcommand{\sD}{\mathscr{D}}
\newcommand{\sE}{\mathscr{E}}
\newcommand{\sF}{\mathscr{F}}
\newcommand{\sG}{\mathscr{G}}
\newcommand{\sH}{\mathscr{H}}
\newcommand{\sI}{\mathscr{I}}
\newcommand{\sJ}{\mathscr{J}}
\newcommand{\sK}{\mathscr{K}}
\newcommand{\sL}{\mathscr{L}}
\newcommand{\sM}{\mathscr{M}}
\newcommand{\sN}{\mathscr{N}}
\newcommand{\sO}{\mathscr{O}}
\newcommand{\sP}{\mathscr{P}}
\newcommand{\sQ}{\mathscr{Q}}
\newcommand{\sR}{\mathscr{R}}
\newcommand{\sS}{\mathscr{S}}
\newcommand{\sT}{\mathscr{T}}
\newcommand{\sU}{\mathscr{U}}
\newcommand{\sV}{\mathscr{V}}
\newcommand{\sW}{\mathscr{W}}
\newcommand{\sX}{\mathscr{X}}
\newcommand{\sY}{\mathscr{Y}}
\newcommand{\sZ}{\mathscr{Z}}
\newcommand{\bZ}{\mathbb{Z}}
\newcommand{\bN}{\mathbb{N}}
\newcommand{\bQ}{\mathbb{Q}}
\newcommand{\bC}{\mathbb{C}}
\newcommand{\bR}{\mathbb{R}}
\newcommand{\bH}{\mathbb{H}}
\newcommand{\bD}{\mathbb{D}}
\newcommand{\bE}{\mathbb{E}}
\newcommand{\bV}{\mathbb{V}}
\newcommand{\bfM}{\mathbf{M}}
\newcommand{\bfN}{\mathbf{N}}
\newcommand{\bfX}{\mathbf{X}}
\newcommand{\bfY}{\mathbf{Y}}
\newcommand{\spec}{\textrm{Spec}}
\newcommand{\dbar}{\bar{\partial}}
\newcommand{\redref}{{\color{red}ref}}

\title[MacPherson's Conjecture] {MacPherson's Conjecture via H\"ormander Estimate}

\author[Junchao Shentu]{Junchao Shentu}
\email{stjc@ustc.edu.cn}
\address{School of Mathematical Sciences,
	University of Science and Technology of China, Hefei, 230026, China}
\author[Chen Zhao]{Chen Zhao}
\email{czhao@ustc.edu.cn}
\address{School of Mathematical Sciences,
	University of Science and Technology of China, Hefei, 230026, China}

\begin{abstract}
In this notes we reprove MacPherson's conjecture on $L^2-(n,q)$-cohomology through Demailly's formulation of H\"ormander's Estimate. This approach allows us to weaken the condition of locally semipositivity in Ruppenthal's $L^2$-representation of adjoint bundle. Moreover we prove the MacPherson's conjecture of twisted coefficient bundle under an arbitrary singular hermitian metric. As applications, we study the MacPherson type problem for pluri-canonical bundle.
\end{abstract}

\maketitle

\section{Introduction}
The $L^2$ technique has become an important part of complex geometry since the works of L. H\"ormander \cite{Hormander1965} and Andreotti-Vesentini \cite{AV1965}. Whereas the theory is well established on complex manifolds, not much is known on singular complex analytic spaces. Motivated by Cheeger-Goresky-MacPherson's conjectural isomorphism between the $L^2$-cohomology and the intersection cohomology of a projective variety \cite{CGM1982}, R. MacPherson \cite{MacPherson1983} conjectured the birational invariance property of $L^2$-arithmetic genus:
\begin{align}\label{align_Mac_conj}
\chi_{(2)}(X_{\rm reg}):=\sum_{q=0}^n(-1)^q\dim H^{0,q}_{(2)}(X_{\rm reg},ds^2_{\rm FS})=\chi(M)
\end{align}
where $\pi:M\to X$ is any resolution of singularities. This conjecture is proved by Pardon-Stern using Donnelly-Fefferman's estimate.
\begin{thm}[Pardon-Stern \cite{Pardon_Stern1991}]\label{thm_PS}
	If $X$ is a complex projective variety of pure dimension $n$ and $X_{\rm reg}$ is given the hermitian metric induced by the embedding of $X$ into the projective space, then for any resolution of singularities $\pi:M\to X$, for all $q$ there is a natural isomorphism
	$$H^{n,q}(M)\simeq H^{n,q}_{(2),\max}(X_{\rm reg}).$$
\end{thm}
(\ref{align_Mac_conj}) follows from Theorem \ref{thm_PS} by the $L^2$-Serre duality (c.f. \cite[Proposition 1.3]{Pardon_Stern1991})
$$H^{n,q}_{(2),\max}(X_{\rm reg})\simeq H^{0,n-q}_{(2),\min}(X_{\rm reg})^\ast.$$
This theorem is later generalized to compact hermitian spaces with a coefficient line bundle by J. Ruppenthal.
\begin{thm}[Ruppenthal \cite{Ruppenthal2014}]\label{thm_Ruppenthal}
	Let $X$ be a compact hermitian complex analytic space of pure dimension $n$. Let $\pi:M\to X$ be a resolution of singularities and $L$ be a holomorphic line bundle on $M$ which is locally semi-positive with respect to $X$. Then the pullback of forms induces the natural isomorphisms
	$$H^{n,q}(M,L)\simeq H^{n,q}_{(2),\max}(X_{\rm reg}, L|_{X_{\rm reg}})\quad \forall q\geq0.$$
\end{thm}
A line bundle $L$ on $M$ is called locally semi-positive with respect to $X$ if locally there is an open subset $U\subset X$ such that $L$ admits a smooth hermitian metric on $\pi^{-1}U$ with semi-positive curvature. We will show that this condition can be removed on the level of complex of sheaves and can be weaken on the level of cohomology.

The main observation of this notes is that Demailly's formulation \cite{Demailly1985,Demailly1989} of H\"ormander's estimate, which plays crucial roles in complex geometry in decades, implies MacPherson's conjecture and its further generalizations. The main result is
\begin{thm}[=Theorem \ref{thm_main_local1}]\label{thm_main_local}
	Let $X$ be a complex analytic space of pure dimension $n$ and $ds^2$ be a hermitian metric on a dense Zariski open subset $X^o\subset X_{\rm reg}$ with $\omega$ its fundamental form. Let $\pi:M\to X$ be a proper holomorphic map such that $M$ is smooth and $\pi:\pi^{-1}X^o\to X^o$ is biholomorphic. Let $(L,h)$ be a holomorphic line bundle with a (possibly) singular hermitian metric $h$ on $M$. Denote by $\sI(h)$ the associated multiplier ideal sheaf. Assume that for every point $x\in X$, there is a neighborhood $x\in U$, a K\"ahler metric $g_U$ on $U\cap X^o$, a (possibly) singular hermitian metric $h'\sim h$ on $L|_{\pi^{-1}U}$ and a bounded $C^\infty$ strictly plurisubharmonic function $\Phi$ on $U\cap X^o$ such that $$g_U\sim\omega|_{U\cap X^o}\lesssim\sqrt{-1}\partial\dbar\Phi$$ on $U$ and 
	\begin{align*}
	\sqrt{-1}\Theta_{h'}(L)\geq C\omega
	\end{align*} 
	on $\pi^{-1}(U)$ as currents for some $C\in\bR$.
	Then the complex of sheaves
	\begin{align*}
	0\to\pi_\ast(K_M\otimes L\otimes \sI(h))\to \sD^{n,0}_{X,ds^2}(L,h)\stackrel{\dbar}{\to}\sD^{n,1}_{X,ds^2}(L,h)\stackrel{\dbar}{\to}\cdots\stackrel{\dbar}{\to} \sD^{n,n}_{X,ds^2}(L,h)\to0
	\end{align*}
	is exact. Here $\sD^{n,\bullet}_{X,ds^2}(L,h)$ denotes the complex consisting of measurable $L|_{X^o}$-valued $(n,\bullet)$-forms $\alpha$ such that $\alpha$ and its distributive $\dbar\alpha$ are locally square integrable.
\end{thm}
\begin{thm}\label{thm_main_global}
	Notations as in Theorem \ref{thm_main_local}. Assume that 
	\begin{enumerate}
		\item $X$ is compact and locally near each point $x\in X$ there is a neighborhood $U$ and a hermitian metric $ds^2_0$ on $U$ such that $ds^2_0\lesssim ds^2|_{U}$,
		\item $\sqrt{-1}\Theta_h(L)$ is $\pi$-relatively semipositive. 
	\end{enumerate}
	We have an isomorphism
	$$H^{n,q}(M,L\otimes\sI(h))\simeq H^{n,q}_{(2),\rm max}(X^o,L|_{X^o})\quad \forall q\geq0.$$
\end{thm}
A $(1,1)$-current $\alpha$ on $M$ is called $\pi$-relatively semipositive if for some positive $(1,1)$-form $\beta$ on $X$, $\beta+\pi^\ast\alpha$ is semipositive. A typical example of $(L,h)$ with $\pi$-relatively semipositive is $(\pi^\ast F,\pi^\ast h)$ where $(F,h)$ is a holomorphic line bundle with a singular metric on $X$.

There are several natural metrics that satisfy the conditions in Theorem \ref{thm_main_local}:

{\bf Hermitian metric}
One typical example of the metrics that satisfy the conditions in  Theorem \ref{thm_main_local} is the hermitian metric on $X$ (metrics that are hermitian at singular points, see Definition \ref{defn_hermitian_metric}). Among many other metrics considered in Theorem \ref{thm_main_local}, metrics that are hermitian on the singularities share better properties. In this case we have
\begin{cor}[=Corollary \ref{cor_main_hermitian_metric1}]\label{cor_main_hermitian_metric}
	Let $X$ be a complex analytic space of pure dimension $n$ and $ds^2$ be a hermitian metric on $X$. Let $\pi:M\to X$ be a resolution of singularities. Let $(L,h)$ be a holomorphic line bundle with a possibly singular hermitian metric $h$ on $M$. Denote by $\sI(h)$ the associated multiplier ideal sheaf. 
	Then the complex of sheaves
	\begin{align*}
	0\to\pi_\ast(K_M\otimes L\otimes \sI(h))\to \sD^{n,0}_{X,ds^2}(L,h)\stackrel{\dbar}{\to}\sD^{n,1}_{X,ds^2}(L,h)\stackrel{\dbar}{\to}\cdots\stackrel{\dbar}{\to} \sD^{n,n}_{X,ds^2}(L,h)\to0
	\end{align*}
	is exact.
	Moreover if $X$ is compact and $\sqrt{-1}\Theta_h(L)$ is $\pi$-relatively semipositive, we have an isomorphism
	$$H^{n,q}(M,L\otimes\sI(h))\simeq H^{n,q}_{(2),\rm max}(X_{\rm reg},L|_{X_{\rm reg}})\quad \forall q\geq0.$$
\end{cor}
When $h$ is  smooth, Corollary \ref{cor_main_hermitian_metric} is reduced to Theorem \ref{thm_Ruppenthal}, thus implies MacPherson's conjecture. 

Let $(F,h_F)$ be a holomorphic line bundle on $X$ with a singular hermitian metric, then $(L,h):=(\pi^\ast F,\pi^\ast h_F)$ has $\pi$-relatively semipositive curvature current.
By the Grothendieck duality and the $L^2$-Serre duality (\ref{prop_Serre_dual}), we also obtain an isomorphism
\begin{align}
H^q_{(2),\min}(X_{\rm reg},F^\vee|_{X_{\rm reg}})\simeq {\rm Ext}^q_M(\sI(\pi^\ast h),\pi^\ast F^\vee)\quad \forall q\geq0.
\end{align}
This implies the generalization of (\ref{align_Mac_conj}):
\begin{align}
\sum_{q=0}^n(-1)^q\dim H^{0,q}_{(2),{\rm min}}(X_{\rm reg},F^\vee|_{X_{\rm reg}})=\sum_{q=0}^n(-1)^q\dim {\rm Ext}^q_M(\sI(\pi^\ast h),\pi^\ast F^\vee).
\end{align}
Pardon-Stern and Ruppenthal use Donnelly-Fefferman-type estimate in their arguments. This estimate requires the condition that locally  the metric should be quasi-isometric to $\sqrt{-1}\partial\dbar\Phi$ where $|\partial\Phi|$ is bounded. Our technique is to use Demailly-H\"ormander's estimate (Proposition \ref{prop_Hormander_incomplete}) together with the requirement that $\Phi$ is bounded. This allows us to add positivity on $L$ without changing the quasi-isometric class of $h$. 
As is stated in Theorem \ref{thm_main_local}, we do not require $|\partial\Phi|$ to be bounded. 

MacPherson's conjecture and Pardon-Stern's solution establish an analytic understanding of the Grauert-Riemenschneidner sheaf for arbitrary complex analytic singularities. Further, it is interesting to consider  the pluri-canonical bundle which plays crucial rules in birational geometry.
As stated above, we do not require the twisted bundle $L$ to be locally semi-positive over $X$. This enables us to study some line bundles $L$ that do not come from the base $X$,  $K_X^{\otimes m}$ for example. 

Recall that $\pi:M\to X$ is a good resolution if ${\rm Exc}(\pi):=\pi^{-1}(X_{\rm sing})$ is a simple normal crossing divisor.
\begin{cor}\label{cor_main_plurican}
	Let $(X,ds^2)$ be a compact hermitian complex analytic space of pure dimension $n$ and $\pi:M\to X$ be a good resolution of singularities. Let $L$ be a line bundle on $M$ with a smooth hermitian metric. Then for each $m\geq 1$ there are isomorphisms
	$$H^{q}_{(2),\rm max}\left(X_{\rm reg}, K_{X_{\rm reg}}^{\otimes m}\otimes L|_{X_{\rm reg}}\right)\simeq H^{q}\left(X,\pi_\ast (K_M^{\otimes m}\otimes\sO_M(\sum_{i\in I}(1-m)a_i D_i)\otimes L)\right),\quad \forall q\geq0.$$
	Here ${\rm Exc}(\pi)=\cup_{i\in I}D_i$ is the irreducible decomposition and $a_i\in \bZ_{\geq 0}$ only depends on $\pi$ and the local quasi-isometric class of $ds^2$.
	
	Assume moreover that $X$ admits only Gorestein singularities so that its dualizing sheaf $\omega_X$ satisfies
	$$\pi^\ast \omega_X\simeq K_M\otimes \sO_M(\sum_{i\in I}b_i D_i)$$
	for some $b_i\in \bZ_{\geq 0}$. Let $F$ be a holomorphic line bundle admits a smooth hermitian metric. Then for every $q\geq0$ there is an isomorphism
	$$H^{q}_{(2),\rm max}\left(X_{\rm reg}, K_{X_{\rm reg}}^{\otimes m}\otimes F|_{X_{\rm reg}}\right)\simeq H^{q}\left(X,\omega_X^{\otimes m}\otimes\pi_\ast\sO_M(\sum_{i\in I}(1-m)(a_i+b_i) D_i)\otimes F\right).$$
\end{cor}
This is a generalization of Theorem \ref{thm_Ruppenthal} to pluri-canonical bundles. Distinguished from Theorem \ref{thm_Ruppenthal}, we see that the $L^2$-pluri-canonical cohomology contains interesting information of the singularities of $X$.

{\bf Conical metric} Another interesting metric that works for Theorem \ref{thm_main_local} is the conical metric.
Let $(\overline{X},D)$ be a log smooth log canonical pair, i.e.
$\overline{X}$ is a compact complex manifold of dimension $n$ and $D=\sum_{i=1}^r b_iD_i\subset \overline{X}$ is an $\bR$-divisor with normal crossing support such that $0<b_i< 1$. $ds_X^2$ is a metric on $X=\overline{X}\backslash D$ which has conical singularities along $D$, i.e. around each point $p\in D$ there is a chart $(z_1,\cdots,z_n)$ such that $D=b_1\{z_1=0\}+\cdots+b_r\{z_r=0\}$ and 
\begin{align}\label{align_cone_metric}
ds_X^2\sim\sum_{i=1}^r|z_i|^{-2b_i}dz_id\overline{z_i}+\sum_{i=r+1}^ndz_id\overline{z_i},\quad b_1,\dots,b_r\in(0,1).
\end{align}
Such metric has attracted many attentions in recent years. They appear as canonical metrics on quasi-projective smooth algebraic varieties with simple boundary asymptotic behavior \cite{Guenancia_Paun2016} and are used to construct K\"ahler Einstein metrics on K-stable Fano manifolds \cite{CDS2015_1,CDS2015_2,CDS2015_3,Tiangang2015}.
This kind of metric has the bounded local potential $\Phi(z)=\sum_{i=1}^r|z_i|^{2-2b_i}$. Thus Theorem \ref{thm_main_local} applies by taking $\pi={\rm Id}$.
\begin{cor}
	Let $(\overline{X},D=\sum_{i=1}^r b_iD_i)$ be a log smooth log canonical pair and $ds^2$ be a hermitian metric with conical singularities associated to $D$. Let $(L,h)$ be a holomorphic line bundle with a smooth hermitian metric on $\overline{X}$. Then there are isomorphisms
	$$H^{n,q}(\overline{X},L)\simeq H^{n,q}_{(2),\rm max}(X,L),\quad \forall q\geq0.$$
\end{cor}
\section{Preliminary on $L^2$-cohomology and H\"ormander's Estimate}
Let us start with the definition of singular hermitian metric.
\begin{defn}
	A function $\varphi:\Omega\to[-\infty,\infty)$ on a complex manifold $\Omega$ is called plurisubharmonic (psh) if $\varphi$ is upper semicontinuous and $\sqrt{-1}\partial\dbar\varphi\geq 0$ as a current.
\end{defn}
\begin{defn}\label{defn_singular_metric}
	A singular hermitian metric $h$ on a holomorphic line bundle $L$ is a measurable section of $L\otimes \overline{L^\ast}$ such that locally $h=e^{-\varphi}h_0$ where $h_0$ is a smooth hermitian metric and $\varphi$ is a psh function.
\end{defn}
Let $X$ be a hermitian manifold of pure complex dimension $n$ and $(E,h)$ be a holomorphic vector bundle on $X$ with a singular hermitian metric. Denote $\dbar:E\to E\otimes\Omega_X$  the canonical $\dbar$ operator. Let $L^k_{2}(E)$ be the space of square integrable $E$-valued $k$-forms. Denote $\dbar_{\rm max}$ to be the $\dbar$ operator on the domains
$$D^{p,q}_{X,\rm max}(E):=\textrm{Dom}^{p,q}(\dbar_{\rm max})=\{\phi\in L_2^{p,q}(E)|\dbar\phi\in L_2^{p,q+1}(E)\}$$
where $\dbar$ is taken in the sense of distribution and $0\leq i\leq \dim_\bR X$. 

The $L^2$ cohomology $H_{(2),\rm max}^{p,\bullet}(X,E;ds^2,h)$ is defined as the cohomology of the complex
\begin{align}\label{align_L2_dol_cohomology}
D^{p,\bullet}_{X,\rm max}(E):0\rightarrow D^{p,0}_{X,\rm max}(E)\stackrel{\dbar_{\rm max}}{\to}D^{p,1}_{X,\rm max}(E)\stackrel{\dbar_{\rm max}}{\to}\cdots\stackrel{\dbar_{\rm max}}{\to}D^{p,n}_{X,\rm max}(E)\to0.
\end{align}
Denote $\dbar_{\rm min}$ to be the graph closure operator of $\dbar:A_{\rm cpt}^{p,\ast}(E)\to A_{\rm cpt}^{p,\ast+1}(E)$. The $L^2$ cohomology $H_{(2),\rm min}^{p,\bullet}(X,E;ds^2,h)$ is defined as the cohomology of the complex
\begin{align}
D^{p,\bullet}_{X,\rm min}(E):0\rightarrow {\rm Dom}^{p,0}(\dbar_{\rm min})\stackrel{\dbar_{\rm min}}{\to}{\rm Dom}^{p,1}(\dbar_{\rm min})\stackrel{\dbar_{\rm min}}{\to}\cdots\stackrel{\dbar_{\rm min}}{\to}{\rm Dom}^{p,n}(\dbar_{\rm min})\to0.
\end{align}
These two $L^2$-cohomologies are related by the $L^2$-Serre duality
\begin{prop}\label{prop_Serre_dual}\cite[Proposition 1.3]{Pardon_Stern1991}
	When $h$ is smooth, there is a natural isomorphism
	$$H_{(2),\rm max}^{p,q}(X,E;ds^2,h)\simeq H_{(2),\rm min}^{n-q,n-p}(X,E^\ast;ds^2,h^\ast)^\ast$$ for every $p$, $q$ whenever $H_{(2),\rm max}^{p,q}(X,E;ds^2,h)$ is finite dimensional.
\end{prop}
Let $X$ be a complex analytic space and $X^o\subset X_{\rm reg}$ be a dense Zariski open subset of the regular locus $X_{\rm reg}$. Let $ds^2$ be a hermitian metric on $X^o$ and $(E,h)$ be a line bundle with a singular metric on $X^o$.
Let $U\subset X$ be an open subset and define $L_{X,ds^2}^{p,q}(E,h)(U)$ to be the space of locally square integrable measurable $E$-valued $(p,q)$-forms on $U\cap X^o$.
For each $(p,q)$ we define a sheaf $\sD_{X,ds^2}^{p,q}(E,h)$ on $X$ by
$$\sD_{X,ds^2}^{p,q}(E,h)(U):=\{\phi\in L_{X,ds^2}^{p,q}(E,h)(U)|\bar{\partial}_{\rm max}\phi\in L_{X,ds^2}^{p,q+1}(E,h)(U)\}$$
for each open subset $U\subset X$.

Define the $L^2$-Dolbeault complex of sheaves $\sD_{X,ds^2}^{p,\bullet}(E,h)$ as
\begin{align}\label{align_D_complex2}
\sD_{X,ds^2}^{p,0}(E,h)\stackrel{\dbar_{\rm max}}{\to}\sD_{X,ds^2}^{p,1}(E,h)\stackrel{\dbar_{\rm max}}{\to}\cdots\stackrel{\dbar_{\rm max}}{\to}\sD_{X,ds^2}^{p,n}(E,h)\to0.
\end{align}
When $X$ is compact and $\sD_{X,ds^2}^{p,\bullet}(E,h)$ is a complex of fine sheaves, we have
\begin{align}
\bH^q(X,\sD_{X,ds^2}^{p,\bullet}(E,h))\simeq H^{p,q}_{(2)}(X,E;ds^2,h).
\end{align}
Next we introduce Demailly's formulation of H\"ormander's estimate over an arbitrary K\"ahler metric. The readers may consult \cite{Demailly2012} for details.
\begin{prop}\label{prop_Hormander_smooth}
	Let $(Y,\omega)$ be a complete K\"ahler manifold of dimension $n$. Let $(L,h)$ be a line bundle with a smooth metric such that $T=[\sqrt{-1}\Theta_h(L),\Lambda_\omega]>0$ on the space of $L-$valued smooth $(p,q)-$forms $A^{p,q}_Y\otimes L$. Then for every $\alpha\in L^{p,q}_{(2)}(Y,L;\omega,h)$, $q\geq1$ such that $\dbar\alpha=0$, there is $\beta\in L^{n,q-1}_{(2)}(Y,L;\omega,h)$ such that $\dbar\beta=\alpha$ and 
	$$\int_Y|\beta|^2_{h}{\rm vol}_\omega\leq \int_Y(T^{-1}\alpha,\alpha)_h{\rm vol}_\omega.$$
\end{prop}
The key technique that we use to prove Theorem \ref{thm_main_local} is the following H\"ormander type estimate, which is proved by Demailly by taking the limit of Proposition \ref{prop_Hormander_smooth}.
\begin{prop}\label{prop_Hormander_incomplete}
	Let $Y$ be a complex manifold of dimension $n$ which admits a complete K\"ahler metric. Let $(L,h)$ be a line bundle with a singular metric such that $\sqrt{-1}\Theta_h(L)\geq \omega$ for some K\"ahler form $\omega$ on $Y$. Then for every $\alpha\in L^{n,q}_{(2)}(Y,L;\omega,h)$, $q\geq1$ such that $\dbar\alpha=0$, there is $\beta\in L^{n,q-1}_{(2)}(Y,L;\omega,h)$ such that $\dbar\beta=\alpha$ and $\|\beta\|^2\leq q\|\alpha\|^2$.
\end{prop}
\begin{defn}
	Let $(Y,ds^2)$ be a hermitian manifold and $(L,h)$ be a holomorphic line bundle with a singular hermitian metric $h$. The multiplier ideal sheaf $\sI(h)\subset\sO_Y$ associated to $h$ is the ideal sheaf consisting of holomorphic functions $f$ such that $fe^{-\varphi}$ is locally square integrable. Here $\varphi$ is locally defined by $h=h_0e^{-\varphi}$ for some smooth hermitian metric $h_0$.
\end{defn}
$\sI(h)$ is independent of the choice of $\varphi$, hence it is a well defined ideal sheaf. Moreover, it is a coherent ideal sheaf. By  definition the holomorphic sections $K_Y\otimes L\otimes \sI(h)$ consists of locally square integrable sections in $K_Y\otimes L$.
\section{$L^2$-resolution of the twisted Grauert-Riemenschneider canonical sheaf}
In this section we prove Theorem \ref{thm_main_local}.
\begin{thm}\label{thm_main_local1}
	Let $X$ be a complex analytic space of pure dimension $n$ and $ds^2$ be a hermitian metric on a dense Zariski open subset $X^o\subset X_{\rm reg}$ with $\omega$ its fundamental form. Let $\pi:M\to X$ be a proper holomorphic map such that $M$ is smooth and $\pi:\pi^{-1}X^o\to X^o$ is biholomorphic. Let $(L,h)$ be a holomorphic line bundle with a (possibly) singular hermitian metric $h$ on $M$. Denote by $\sI(h)$ the associated multiplier ideal sheaf. Assume that for every point $x\in X$, there is a neighborhood $x\in U$, a K\"ahler metric $g_U$ on $U\cap X^o$, a (possibly) singular hermitian metric $h'\sim h$ on $L|_{\pi^{-1}U}$ and a bounded $C^\infty$ strictly plurisubharmonic function $\Phi$ on $U\cap X^o$ such that $$g_U\sim\omega|_{U\cap X^o}\lesssim\sqrt{-1}\partial\dbar\Phi$$ on $U$ and 
	\begin{align}
	\sqrt{-1}\Theta_{h'}(L)\geq C\omega
	\end{align} 
	on $\pi^{-1}(U)$ as currents for some $C\in\bR$.
	Then the complex of sheaves
	\begin{align}\label{align_fine_resolution}
	0\to\pi_\ast(K_M\otimes L\otimes \sI(h))\to \sD^{n,0}_{X,ds^2}(L,h)\stackrel{\dbar}{\to}\sD^{n,1}_{X,ds^2}(L,h)\stackrel{\dbar}{\to}\cdots\stackrel{\dbar}{\to} \sD^{n,n}_{X,ds^2}(L,h)\to0
	\end{align}
	is exact. Here $\sD^{n,\bullet}_{X,ds^2}(L,h)$ denotes the complex consisting of measurable $L|_{X^o}$-valued $(n,\bullet)$-forms $\alpha$ such that $\alpha$ and its distributive $\dbar\alpha$ are locally square integrable.
\end{thm}
\begin{proof}
	By assumption, there is $C'>0$ such that $C'\sqrt{-1}\partial\dbar\Phi\geq \omega|_{U\cap X^o}$. Let $h''=e^{C'(C-1)\Phi}h'$. Since $\Phi$ is bounded, we have $h''\sim h'\sim h$ and
	$$\sqrt{-1}\Theta_{h''}(L)=(1-C)C'\sqrt{-1}\partial\dbar\Phi+\sqrt{-1}\Theta_{h'}(L)\geq \omega.$$
	By Lemma \ref{lem_complete_metric_exists_locally} we may assume that $U\cap X^o$ admits a complete K\"ahler metric. By Proposition \ref{prop_Hormander_incomplete}, we have
	$$H^q_{(2)}(U\cap X^o, L|_{U\cap X^o};ds^2,h)\simeq H^q_{(2)}(U\cap X^o, L|_{U\cap X^o};g_U,h'')=0,\quad \forall q>0.$$
	This proves the exactness of (\ref{align_fine_resolution}) at $\sD^{n,q}_{X,ds^2}(L,h)$, $q>0$. It remains to show that 
	\begin{align}\label{align_exact_at_0}
	\pi_\ast(K_M\otimes L\otimes \sI(h))=\ker\left(\dbar:\sD^{n,0}_{X,ds^2}(L,h)\to\sD^{n,1}_{X,ds^2}(L,h)\right).
	\end{align}
	Let $V\subset X$ be an open subset. Since $\pi$ is a proper morphism, it suffices to show that a holomorphic section $s\in \Gamma(\pi^{-1}V,K_M\otimes L)$ is locally square integrable  with respect to the (possibly degenerate) metric $\pi^\ast ds^2$ and $h$, if and only if it lies in $\Gamma(\pi^{-1}V,K_M\otimes L\otimes \sI(h))$. Let $V'\subset \pi^{-1}V$ be an open subset such that there is an orthogonal frame of cotangent fields $\delta_1,\dots,\delta_n$ such that 
	\begin{align}\label{align_metric1}
	\pi^\ast ds^2=\lambda_1\delta_1\overline{\delta_1}+\cdots+\lambda_n\delta_n\overline{\delta_n}.
	\end{align}
	Denote by $g$  a K\"ahler form on $M$, then on $V'$ we have
	\begin{align}\label{align_metric2}
	g\sim\delta_1\overline{\delta_1}+\cdots+\delta_n\overline{\delta_n}.
	\end{align}
	Denote $\gamma=\delta_1\wedge\cdots\wedge\delta_n\otimes\xi$.
	By (\ref{align_metric1}) and (\ref{align_metric2}) we obtain
	\begin{align}
	\|\gamma|_{V'}\|^2_{\pi^\ast ds^2,h}&=\int_{V'}|\delta_1\wedge\cdots\wedge\delta_n\otimes\xi|^2_{\pi^\ast ds^2,h}\prod_{i=1}^n\lambda_i\delta_i\wedge\overline{\delta_n}\\\nonumber
	&=\int_{V'}|\xi|^2_{h}\prod_{i=1}^n\delta_i\wedge\overline{\delta_n}\\\nonumber
	&\sim \|\gamma|_{V'}\|^2_{g,h}.
	\end{align}
	This shows that $\|\gamma|_{V'}\|^2_{\pi^\ast ds^2,h}<\infty$ if and only if $\|\gamma|_{V'}\|^2_{g,h}<\infty$, where the latter is equivalent to $\gamma\in \Gamma(V',K_M\otimes L\otimes \sI(h))$. This proves (\ref{align_exact_at_0}). Hence we obtain the first part of Theorem \ref{thm_main_local}. To prove the second part, it suffices to show that $\sD^{n,\bullet}_{X,ds^2}(L,h)$ is a complex of fine sheaves. This is proved in Lemma \ref{lem_fine_sheaf}.
\end{proof}

\begin{lem}\label{lem_complete_metric_exists_locally}
	Let $x\in X$ be a point in a complex analytic space and $X^o\subset X_{\rm{reg}}$ be a dense Zariski open subset. There is a neighborhood $x\in U$ and a complete K\"ahler metric on $U\cap X^o$.
\end{lem}
\begin{proof}
	The construction follows that in \cite{Pardon_Stern1991}. Take an open neighborhood $U$ which is embedded holomorphically into $\bC^N$, such that $x=0\in \bC^N$ and $U=B_c:=\left\{z\in\bC^N\big||z|<c\right\}$. Assume that $X^o\cap B_c$ is defined by holomorphic functions $f_1,\dots,f_m\in\sO(B_c)$ and $0<c$ is small enough so that
	$$\sum_{i=1}^m|f_i|^2\ll1.$$ 
	We set
	$$G=-\log(c^2-|z|^2)-\log\left(-\log\sum_{i=1}^m|f_i|^2\right).$$
	Then $\sqrt{-1}\partial\dbar G$ is a complete K\"ahler metric on $U\cap X^o$ (c.f. \cite{Pardon_Stern1991} Lemma 2.4).
\end{proof}
The following lemma gives a sufficient condition which guarantee that $\sD^{p,q}_{X,ds^2}(L,h)$ is a fine sheaf.
\begin{lem}\label{lem_fine_sheaf}
	Let $X$ be an analytic space and $ds^2$ be a hermitian metric on $X_{\rm reg}$. Suppose that $ds^2_0\lesssim ds^2$ where $ds^2_0$ is a hermitian metric on $X$ (Definition \ref{defn_hermitian_metric}). Then for each $p$, $q$, $\sD^{p,q}_{X,ds^2}(L,h)$ is a fine sheaf.
\end{lem}
\begin{proof}
	If suffices to show that for every $W\subset\overline{W}\subset U\subset X$ where $W$ and $U$ are open subsets, there is a continuous function $f$ on $U$ such that
	\begin{itemize}
		\item ${\rm supp}(f)\subset \overline{W}$,
		\item $f$ is $C^\infty$ on $U_{\rm reg}$
		\item $\dbar f$ has bounded fiberwise norm with respect to the metric $ds^2_0$.
	\end{itemize}
	Choose a closed embedding $U\subset M$ where $M$ is a smooth complex manifold $M$. Let $V\subset\overline{V}\subset M$ where $V$ is an open subset such that $V\cap U=W$. Let $ds^2_M$ be a hermitian metric on $M$ so that $ds^2_0|_{U_{\rm reg}}\sim ds^2_M|_{U_{\rm reg}}$. Let $g$ be a smooth function on $M$ whose support lies in $\overline{V}$. Denote $f=g|_U$ then apparently ${\rm supp}(f)\subset \overline{W}$ and $f$ is $C^\infty$ on $U_{\rm reg}$. It suffices to show the boundedness of the fiberwise norm of $\dbar f$. Since $U_{\rm reg}\subset M$ is a submanifold, one has the orthogonal decomposition 
	\begin{align}
	T_{M,x}=T_{U_{\rm reg},x}\oplus T_{U_{\rm reg},x}^\bot, \quad \forall x\in U_{\rm reg}.
	\end{align} 
	Therefore $|\dbar f|_{ds^2}\lesssim|\dbar f|_{ds^2_0}\leq |g|_{ds^2_M}<\infty$. The lemma is proved.
\end{proof}
We are ready to prove Theorem \ref{thm_main_global}
\begin{thm}\label{thm_main_global1}
	Notations as in Theorem \ref{thm_main_local}. Assume that 
	\begin{enumerate}
		\item $X$ is compact and locally near each point $x\in X$ there is a neighborhood $U$ and a hermitian metric $ds^2_0$ on $U$ such that $ds^2_0\lesssim ds^2|_{U}$,
		\item $\sqrt{-1}\Theta_h(L)$ is $\pi$-relatively semipositive. 
	\end{enumerate}
	We have an isomorphism
	$$H^{n,q}(M,L\otimes\sI(h))\simeq H^{n,q}_{(2),\rm max}(X^o,L|_{X^o})\quad \forall q\geq0.$$
\end{thm}
\begin{proof}
	By condition (1) and lemma \ref{lem_fine_sheaf}, $\sD^{n,\bullet}_{X,ds^2}(L,h)$ is a complex of fine sheaves. Therefore Theorem \ref{thm_main_local1} gives isomorphisms
	\begin{align}\label{align_aaa}
	H^{q}\left(X,\pi_\ast\left(K_M\otimes L\otimes\sI(h)\right)\right)\simeq H^{n,q}_{(2),\rm max}(X^o,L|_{X^o})\quad \forall q\geq0.
	\end{align}
	We are going to show that
	\begin{align}\label{align_GR_vanishing}
	R^q\pi_\ast\left(K_M\otimes L\otimes\sI(h)\right)=0,\quad q>0.
	\end{align}
	Since the problem is local, we may assume that $X$ admits a bounded smooth strictly psh function $\Phi$. Indeed we can embedded $X$ into a complex manifold $V$ and restrict a bounded smooth strictly psh function on $X$. After a possibly shrinking there is a constant $C>0$ such that
	\begin{align}\label{align_geq0_h'}
	\sqrt{-1}C\partial\dbar\Phi'+\sqrt{-1}\Theta_h(L)\geq 0.
	\end{align}
	Let $h'=e^{-C\Phi'}h$ be a singular metric on $L$. Since $\Phi'=\pi^\ast\Phi$ is bounded, we obtain $h'\sim h$ and hence $\sI(h')=\sI(h)$. By (\ref{align_geq0_h'}) we have
	$\sqrt{-1}\Theta_{h'}(L)\geq 0$. Now (\ref{align_GR_vanishing}) follows from \cite[Corollary 1.5]{Matsumura2018}.
	As a consequence of (\ref{align_GR_vanishing}) the Leray spectral sequence 
	$$H^p\left(X,R^q\pi_\ast\left(K_M\otimes L\otimes\sI(h)\right)\right)\Rightarrow H^{p+q}\left(M,K_M\otimes L\otimes\sI(h)\right)$$
	degenerates at the $E_1$-page and we obtain a natural isomorphism
	\begin{align}
		H^{q}\left(X,\pi_\ast\left(K_M\otimes L\otimes\sI(h)\right)\right)\simeq H^q(M,K_M\otimes L\otimes\sI(h)),\quad q\geq 0.
	\end{align}
	Combining this with (\ref{align_aaa}) we prove the theorem.
\end{proof}
\section{Applications}
\subsection{Hermitian Metrics}
\begin{defn}\label{defn_hermitian_metric}
	Let $X$ be a complex analytic space and $ds^2$ be a hermitian metric on $X_{\rm reg}$. We say $ds^2$ is a hermitian metric on $X$ if for every $x\in X$ there is a neighborhood $U$ and a holomorphic closed immersion $U\subset V$ into a holomorphic manifold such that $ds^2|_U\sim ds^2_V|_{U}$ for some hermitian metric $ds^2_V$ on $V$. 
\end{defn}
Let $X$ be a complex analytic space. Recall that a holomorphic map $\pi:M\to X$ is a resolution of singularities if the followings are valid.
\begin{itemize}
	\item $M$ is smooth,
	\item $\pi$ is proper,
	\item $\pi:\pi^{-1}X_{\rm reg}\to X_{\rm reg}$ is biholomorphic,
\end{itemize}
A resolution of singularities $\pi:M\to X$ is called good if  ${\rm Exc}(\pi):=\pi^{-1}(X_{\rm sing})$ is a simple normal crossing divisor.

The following simple lemma is  the property that hermitian metrics are distinguished from many others.
\begin{lem}\label{lem_lemma_hermitian_metric}
	Let $X$ be a complex analytic space and $ds^2$ be a hermitian metric on $X$. Assume that $ds^2|_{X_{\rm reg}}$ is K\"ahler and denote $\omega$ to be its K\"ahler form. Let $\pi:M\to X$ be a resolution of singularities and $\alpha$ be a smooth $(1,1)$-form on $M$. Then locally there is a constant $C>0$ such that $C\pi^\ast\omega+\alpha$ is positive over $\pi^\ast X_{\rm reg}$.
\end{lem}
\begin{proof}
	Let $\lambda_1,\dots,\lambda_n$ be the (local) eigenvalue functions of $\alpha$. Since $\alpha$ is smooth, $\lambda_1,\dots,\lambda_n$ are bounded. Because $\omega$ is positive (even at $X_{\rm sing}$), there is a constant $C>0$ such that $C\omega+\pi_\ast\alpha$ is positive on $X_{\rm reg}$. Notice that $\pi_\ast\alpha$ is well defined on $X_{\rm reg}$ which is bounded near $X_{\rm sing}$.
\end{proof}
\begin{cor}\label{cor_main_hermitian_metric1}
	Let $X$ be a complex analytic space of pure dimension $n$ and $ds^2$ be a hermitian metric on $X$. Let $\pi:M\to X$ be a resolution of singularities. Let $(L,h)$ be a holomorphic line bundle with a possibly singular hermitian metric $h$ on $M$. Denote by $\sI(h)$ the associated multiplier ideal sheaf. 
	Then the complex of sheaves
	\begin{align*}
	0\to\pi_\ast(K_M\otimes L\otimes \sI(h))\to \sD^{n,0}_{X,ds^2}(L,h)\stackrel{\dbar}{\to}\sD^{n,1}_{X,ds^2}(L,h)\stackrel{\dbar}{\to}\cdots\stackrel{\dbar}{\to} \sD^{n,n}_{X,ds^2}(L,h)\to0
	\end{align*}
	is exact.
	Moreover if $X$ is compact and $\sqrt{-1}\Theta_h(L)$ is $\pi$-relatively semipositive, we have an isomorphism
	$$H^{n,q}(M,L\otimes\sI(h))\simeq H^{n,q}_{(2),\rm max}(X_{\rm reg},L|_{X_{\rm reg}})\quad \forall q\geq0.$$
\end{cor}
\begin{proof}	
	Recall that in Definition \ref{defn_hermitian_metric} we may assume that $ds^2_V$ is K\"ahler since hermitian metrics are locally quasi-isomorphic to a K\"ahler one.	
	Since every K\"ahler metric on a manifold admits bounded potential functions, we see that $ds^2$ satisfies all the  conditions in Theorem \ref{thm_main_local} if it is hermitian on $X$. 
	
	Let $\Phi$ be the bounded local potential function of $ds^2$, i.e. the fundamental form of $ds^2$ is locally quasi-isometric to $\sqrt{-1}\partial\dbar\Phi$. Since $h$ is a singular metric, for every point $x\in X$ there is an open neighborhood $U\subset X$ and a smooth $(1,1)$-form $\alpha$ on $\pi^{-1}U$ such that $$\sqrt{-1}\Theta_{h}(L)\geq\alpha$$
	on $\pi^{-1}U$. By Lemma \ref{lem_lemma_hermitian_metric}, there is a constant $C$ such that
	\begin{align}
	C\pi^\ast\sqrt{-1}\partial\dbar\Phi+\alpha>0.
	\end{align}
	Let $h'=e^{-(1+C)\Phi}h$, then $h'\sim h$ and 
	\begin{align}
	\sqrt{-1}\Theta_{h'}(L)=(C+1)\pi^\ast\sqrt{-1}\partial\dbar\Phi+\sqrt{-1}\Theta_{h}(L)>\pi^\ast\sqrt{-1}\partial\dbar\Phi.
	\end{align}
	Then the  results follow from Theorem \ref{thm_main_local} and Theorem \ref{thm_main_global}.
\end{proof}
We will use this result to study the MacPherson-type problem on pluri-canonical bundles.
\subsection{Pluri-canonical Bundle}
Notice that $K_{X_{\rm reg}}^{\otimes m}$ is endowed with the  metric induced from $ds^2$.
\begin{cor}
	Let $(X,ds^2)$ be a hermitian complex analytic space of pure dimension $n$ and $\pi:M\to X$ be a good resolution of singularities. Let $(L,h_L)$ be a line bundle with a smooth hermitian metric. Let $K_M$ be endowed with the singular metric induced from $ds^2$, then for each $m\geq 1$ there are isomorphisms
	$$H^{q}_{(2),\rm max}\left(X_{\rm reg}, K_{X_{\rm reg}}^{\otimes m}\otimes L|_{X_{\rm reg}}\right)\simeq H^{q}\left(X,\pi_\ast (K_M^{\otimes m}\otimes\sO_M(\sum_{i\in I}(1-m)a_i D_i)\otimes L)\right),\quad \forall q\geq0.$$
	Here ${\rm Exc}(\pi)=\cup_{i\in I}D_i$ is the irreducible decomposition and $a_i\in \bZ_{\geq 0}$ only depends on $\pi$ and the local quasi-isometric class of $ds^2$.
	
	Assume moreover that $X$ admits only Gorestein singularities so that its dualizing sheaf $\omega_X$ satisfies
	$$\pi^\ast \omega_X\simeq K_M\otimes \sO_M(\sum_{i\in I}b_i D_i)$$
	for some $b_i\in \bZ_{\geq 0}$. Let $F$ be a holomorphic line bundle admits a smooth hermitian metric. Then for every $q\geq0$ there is an isomorphism
	$$H^{q}_{(2),\rm max}\left(X_{\rm reg}, K_{X_{\rm reg}}^{\otimes m}\otimes F|_{X_{\rm reg}}\right)\simeq H^{q}\left(X,\omega_X^{\otimes m}\otimes\pi_\ast\sO_M(\sum_{i\in I}(1-m)(a_i+b_i) D_i)\otimes F\right).$$
\end{cor}
\begin{proof}
	Let $\omega$ be the fundamental form associated to $ds^2$. Denote by $h$ the hermitian metric on the canonical bundle $K_{X_{\rm reg}}$ induced from $ds^2$. Notice that $h$ may be regarded as a hermitian metric on $K_M$ which has singularities along ${\rm Exc}(\pi)$. Since $$\sqrt{-1}\Theta_h(K_{X_{\rm reg}})=-{\rm Ric}(\pi^{\ast}ds^2)$$
	and $\pi^\ast ds^2$ is a smooth degenerate metric on $M$, we see that $\Theta_h(K_{M})$ is smooth on $M$. Therefore 
	$$\sqrt{-1}\Theta_{h^{\otimes (m-1)}\otimes h_L}(K_{M}^{\otimes(m-1)}\otimes L)=(m-1)\sqrt{-1}\Theta_h(K_{M})+\sqrt{-1}\Theta_{h_L}(L)$$
	is a smooth form.
	As a consequence of Corollary \ref{cor_main_hermitian_metric} there is a quasi-isomorphism
	\begin{align}
	\pi_\ast\left(K_M^m\otimes L\otimes \sI(h^{\otimes(m-1)})\right)\simeq_{\rm q.i.s.} \sD^{n,\bullet}_{X,ds^2}\left(K_M^{\otimes(m-1)}\otimes L,h^{\otimes(m-1)}\otimes h_L\right).
	\end{align}
	It remains to calculate $\sI(h^{\otimes(m-1)})$.
	
	Let $x\in X$ and $U$ be an open neighborhood with a closed immersion $U\subset\Omega$ into a complex manifold $\Omega$. Without loss of generality we may assume that $ds^2$ is K\"ahler around $x$. Denote $\pi=(\pi_1,\dots,\pi_N):\pi^{-1}U\to \Omega\subset\bC$. Let $(y_1,\dots,y_n)$ be a local holomorphic chart of $y\in\pi^{-1}\{x\}\subset M$. Then there are holomorphic functions $\lambda_{ij}$, $1\leq i\leq N$, $1\leq j\leq n$ such that
	\begin{align}
	d\pi_i=\sum_{j=1}^n\lambda_{ij}dy_j,\quad\forall i=1,\dots,N.
	\end{align}
	Denote by $\omega_{\pi^\ast ds^2}$ and $\omega_M$  the K\"ahler forms associated to the (degenerate) K\"ahler metric $\pi^\ast ds^2$ and a K\"ahler metric $ds^2_M$ near $y$ respectively. Then
	\begin{align}\label{align_adjunction}
	{\rm vol}_{\pi^\ast ds^2}&=\frac{1}{n!}\omega_{\pi^\ast ds^2}^n\\\nonumber
	& = \frac{1}{n!}\Lambda dy_1\wedge\cdots dy_n\wedge\overline{\Lambda}d\overline{y}_1\wedge\cdots d\overline{y}_n\\\nonumber
	& \sim |\Lambda|^2 {\rm vol}_M,
	\end{align}
	where
	\begin{align}
	\Lambda=\sum_{(i_1,\dots,i_n)\in {\bf N}^n}\lambda_{i_11}\cdots\lambda_{i_nn}, \quad {\bf N}=\{1,\dots,N\}
	\end{align}
	is a holomorphic function.
	
	Since the map $\pi:V:=\pi^{-1}U\to\Omega$ is biholomorphic away from ${\rm Exc}(\pi)\cap \pi^{-1}U$, by (\ref{align_adjunction}) $\Lambda$ is invertible over $\pi^{-1}U\backslash {\rm Exc}(\pi)$. Hence $\Lambda$ defines a divisor $\sum_{i\in I}a_i D_i$ supported in ${\rm Exc}(\pi)\cap \pi^{-1}U$. Therefore for every small neighborhood $V$ of $y$ and $\sigma\in \Gamma(V,K_M^m\otimes L)$, we have
	\begin{align}
	\int|\sigma|^2_{\pi^\ast ds^2}{\rm vol}_{\pi^\ast ds^2}\sim \int |\Lambda^{1-m}\sigma|^2_{ds^2_M}{\rm vol}_{ds^2_M}.
	\end{align}
	As a consequence, $\sigma$ is locally square integrable with respect to $\pi^\ast ds^2$ and $h_L$ if and only if $\Lambda^{1-m}\sigma$ is locally square integrable with respect to $ds^2_M$ and $h_L$. This finishes the proof of the first part. The second part follows from Theorem \ref{thm_main_global}.
\end{proof}
\subsection{Generalized Conical Metric}
\begin{cor}
	Let $(\overline{X},D=\sum_{i=1}^r b_iD_i)$ be a log smooth log canonical pair and $ds^2$ be a hermitian metric with conical singularities associated to $D$. Let $(L,h)$ be a holomorphic line bundle with a smooth hermitian metric. Then there are isomorphisms
	$$H^{n,q}(\overline{X},L)\simeq H^{n,q}_{(2),\rm max}(X,L),\quad \forall q\geq0.$$
\end{cor}
\begin{proof}
	By (\ref{align_cone_metric}) the local potential function $\Phi(z)=\sum_{i=1}^r|z_i|^{2-2b_i}$ is bounded. Since $L$ is defined on $M$, locally there is always a hermitian metric with semi-positive curvature. Moreover
	$$\sum_{i=1}^ndz_id\overline{z_i} \lesssim\sum_{i=1}^r|z_i|^{-2b_i}dz_id\overline{z_i}+\sum_{i=r+1}^ndz_id\overline{z_i},\quad b_1,\dots,b_r\in(0,1).$$
	Hence by Theorem \ref{thm_main_local} we obtain the corollary.
\end{proof}

\bibliographystyle{plain}
\bibliography{CGM}

\end{document}